\documentclass[11pt,a4paper]{amsart}
\usepackage{etex} 
\usepackage[latin1]{inputenc}
\usepackage[T1]{fontenc}
\usepackage[right=2.5cm,left=2.5cm,top=3cm,bottom=3cm]{geometry}
\usepackage{ifthen}
\usepackage{enumerate}
\usepackage{amsmath}
\usepackage{amsthm}
\usepackage{amsfonts}
\usepackage{amssymb}
\usepackage{latexsym}
\usepackage{ae}
\usepackage[normalem]{ulem}

\title{Short interval results for certain prime-independent multiplicative functions}
\author{Olivier Bordellès}
\address{2 allée de la combe \\ 43000 Aiguilhe \\ France}
\email{borde43@wanadoo.fr}
\date{}
\dedicatory{}

\newcommand{\Z}{\mathbb {Z}}

\newcommand{\R}{\mathbb {R}}

\DeclareMathOperator{\res}{Res}

\theoremstyle{plain}
\newtheorem{theorem}{Theorem}
\newtheorem{corollary}[theorem]{Corollary}
\newtheorem{lemma}[theorem]{Lemma}

\theoremstyle{definition}

\theoremstyle{remark}

\makeatletter
\@namedef{subjclassname@2010}{%
  \textup{2010} Mathematics Subject Classification}
\makeatother

\begin{document}

\begin{abstract}
Using recent results from the theory of integer points close to smooth curves, we give an asymptotic formula for the distribution of values of a class of integer-valued prime-independent multiplicative functions. 
\end{abstract}


\maketitle

\font\rms=cmr8 
\font\its=cmti8 
\font\bfs=cmbx8

\section{Introduction and result}
\label{s1}

\noindent
A prime-independent multiplicative function is a multiplicative arithmetic function $f$ satisfying $f(1)=1$ and such that there exists a map $g: \Z_{\geqslant 0} \longrightarrow \R$ such that $g(0)=1$ and, for any prime powers $p^\alpha$
$$f \left( p^\alpha \right) = g(\alpha).$$
In this article, we only consider integer-valued prime-independent multiplicative functions $f$ verifying $f(p)=1$ for any prime $p$. This is equivalent to the fact that $g(1)=1$ and we also assume that there exists $r \in \Z_{\geqslant 2}$ such that
\begin{equation}
   g(1) = \dotsb = g(r-1) = 1 \quad \textrm{and} \quad \alpha \geqslant r \Rightarrow g(\alpha) > 1. \label{e1}
\end{equation}
One of the long-standing problems in number theory concerning these prime-independent multiplicative functions is the study of the distribution of their values. To this end, we fix $k \in \Z_{\geqslant 1}$ and set
$$S_{f,k} (x) := \sum_{\substack{n \leqslant x \\ f(n)= k}} 1 $$
and define the \textit{local density} of $f$ to be the real number 
$$d_{f,k} := \lim_{x \to \infty} \frac{S_{f,k} (x)}{x}$$
whenever the limit exists.

\medskip

\noindent
The arithmetic function $n \longmapsto a(n)$, counting the number of finite, non-isomorphic abelian groups of order $n$, is one of the well-known examples of prime-independent multiplicative functions, since $a\left( p^\alpha \right) = P(\alpha)$, where $P$ is the unrestricted partition function. The existence of the local density $d_{a,k}$ was first established in \cite{ken} and later Ivi\'{c} \cite{ivi0} showed that
$$S_{a,k} (x) = d_{a,k} x + O \left( x^{1/2} \log x \right).$$
Further authors improved on this estimate, such as \cite{ivi}, \cite{kraw} and \cite{now} in which the best error term to date was established. The general case was introduced by Ivi\'{c} in \cite{ivi} and improved in \cite{now} for a certain class of arithmetic functions.

\medskip

\noindent
The next step was the study of the distribution of values of $f$ in short intervals. By 'short intervals' we mean the study of sums of the shape
$$S_{f,k} (x+y) - S_{f,k}(x) = \sum_{\substack{x < n \leqslant x+y \\ f(n)=k}} 1$$
where $y=o(x)$ as $x \to \infty$. In the case of $f=a$, Ivi\'{c} \cite{ivi1} first showed that
$$\sum_{\substack{x < n \leqslant x+y \\ a(n)=k}} 1 = d_{a,k} y + o(y)$$
holds for $y \geqslant x^{581/1744} \log x$. This value was successfully improved by many authors. For instance, by connecting the problem to the error term in certain divisor problems, Krätzel \cite{kra} showed that
$$\sum_{\substack{x < n \leqslant x+y \\ a(n)=k}} 1 = d_{a,k} y + o(y) + O \left( x^{369/1667 + \varepsilon} \right).$$
On the other hand, using results on gaps between squarefree numbers, Li \cite{li} proved that the asymptotic formula
$$\sum_{\substack{x < n \leqslant x+y \\ a(n)=k}} 1 = d_{a,k} y + o(y)$$
holds for $y \geqslant x^{1/5 + \varepsilon}$ uniformly for $k \in \Z_{\geqslant 1}$. In the general case, Zhai \cite[Theorem 2.5]{zha} showed that
$$\sum_{\substack{x < n \leqslant x+y \\ f(n)=k}} 1 = d_{f,k} y + o(y)$$
holds for $y \geqslant x^{\frac{1}{2r+1} + \varepsilon}$ where $r$ is given in \eqref{e1}. The purpose of this work is to establish an effective version of Zhai's result by giving a fully effective error term. More precisely, we will show the following estimate.

\begin{theorem}
\label{t1}
Let $k \in \Z_{\geqslant 1}$ fixed and $f$ be an integer-valued prime-independent multiplicative function such that $f(p)=1$ for any prime $p$ and let $r \in \Z_{\geqslant 2}$ as in \eqref{e1}. Let $x^{\frac{1}{2r+1} + \varepsilon} \leqslant y \leqslant 4^{-2r^2} x$ be real numbers. Then
$$\sum_{\substack{x < n \leqslant x+y \\ f(n)=k}} 1 = d_{f,k} y + O_{r,\varepsilon} \left\lbrace \left( x^{r-1} y^{r+1} \right)^{\frac{1}{2r^2}} x^\varepsilon + y x^{- \frac{1}{6(4r-1)(2r-1)} + \varepsilon} + y^{1 - \frac{2(r-1)}{r(3r-1)}} x^\varepsilon \right\rbrace.$$
\end{theorem} 

\section{Notation and preparation for the proof}
\label{s2}

\noindent
In what follows, $k \in \Z_{\geqslant 1}$ is fixed and $f$ is an integer-valued prime-independent multiplicative function satisfying the hypothesis of Theorem~\ref{t1}, with $r \in \Z_{\geqslant 2}$ given in \eqref{e1}.

\medskip

\noindent
For any arithmetic function $F$, $L(s,F)$ is its formal Dirichlet series and $F^{-1}$ is the Dirichlet convolution inverse of $F$.

\medskip

\noindent
Let $s_r$ be the characteristic function of the set of $r$-full numbers, $\mu_r$ be that of the set of $r$-free numbers, so that $\mu_r^{-1}$ is the multiplicative function such that $\mu_r^{-1}(1)=1$ and given on prime powers $p^\alpha$ by
$$\mu_r^{-1} \left( p^\alpha \right)  = \left\lbrace \begin{array}{rl} 1, & \textrm{if\ } r \mid \alpha; \\ -1, & \textrm{if\ } r \mid \alpha - 1; \\ 0, & \textrm{otherwise}. \end{array} \right.$$
Finally, put
$$\mathbf{1}_{f,k}(n) = \begin{cases} 1, & \textrm{if\ } f(n)=k; \\ 0, & \textrm{otherwise}. \end{cases}$$
Note that $f(n)=1$ whenever $n$ is $r$-free so that the Dirichet series of $\mathbf{1}_{f,k}$ may be formally written as
\begin{equation}
   L \left( s,\mathbf{1}_{f,k} \right) = \frac{\zeta(s)}{\zeta(rs)} H_{f,k,r} (s) := \frac{\zeta(s)}{\zeta(rs)} \sum_{n=1}^\infty \frac{h_{f,k,r}(n)}{n^s} \label{e2}
\end{equation}
and where the multiplicative function $h_{f,k,r}$ is supported on $r$-full numbers. Indeed
$$h_{f,k,r}(n) = \sum_{\substack{d \mid n \\ f(n/d) =k}} \mu_r^{-1} (d)$$
which implies that, for any prime powers $p^\alpha$ with $1 \leqslant \alpha < r$
\begin{eqnarray*}
   h_{f,k,r} \left( p^\alpha \right) &=& \sum_{j=0}^\alpha \mathbf{1}_{f,k} \left( p^{\alpha - j} \right) \mu_r^{-1} \left( p^j \right) \\
   &=& \sum_{j=0}^{\lfloor \alpha / r \rfloor} \mathbf{1}_{f,k} \left( p^{\alpha-rj} \right) - \sum_{j=0}^{\lfloor (\alpha-1) / r \rfloor} \mathbf{1}_{f,k} \left( p^{\alpha-rj - 1} \right) \\
   &=& \mathbf{1}_{f,k} \left( p^\alpha \right) - \mathbf{1}_{f,k} \left( p^{\alpha-1} \right) = 0
\end{eqnarray*}
since $g(\alpha) = g(\alpha-1) = 1$. This in turn implies that the Dirichlet series $H_{f,k,r}$ is absolutely convergent in the half-plane $\sigma > \frac{1}{r}$ and also that
\begin{equation}
   \left | h_{f,k,r}(n) \right | \leqslant s_r(n) \tau(n) \label{e3}
\end{equation}
for any $k,n \in \Z_{\geqslant 1}$ and $r \in \Z_{\geqslant 2}$. The following bound will then be useful.

\begin{lemma}
\label{le1}
Let $r \in \Z_{\geqslant 2}$. Then
$$\sum_{n \leqslant x} s_r(n) \tau(n) \ll x^{1/r} (\log x)^r.$$
\end{lemma}

\begin{proof}
Every $r$-full integer $n$ may be uniquely written as $n = a_1^r a_2^{r+1} \dotsb a_r^{2r-1}$ with $a_2 \dotsb a_r$ squarefree and $(a_i,a_j)=1$ for $2 \leqslant i < j \leqslant r$. Since the divisor function $\tau$ is sub-multiplicative, we infer that the sum of the lemma does not exceed
$$\ll \sum_{a_r \leqslant x^{\frac{1}{2r-1}}} \tau \left( a_r^{2r-1} \right) \sum_{a_{r-1} \leqslant \left( \frac{x}{a_r^{2r-1}} \right)^{\frac{1}{2r-2}}} \tau \left( a_{r-1}^{2r-2} \right) \dotsb \sum_{a_1 \leqslant \left( \frac{x}{a_2^{r+1} \dotsb a_r^{2r-1}} \right)^{1/r}} \tau \left( a_1^{r} \right).$$
Now the well-known bound
$$\sum_{a \leqslant z} \tau \left( a^r \right) \ll z (\log z)^r$$
applied to the last inner sum, allows us to complete the proof.
\end{proof}

\noindent
The next result is an immediate consequence of Lemma~\ref{le1}.

\begin{lemma}
\label{le2}
Let $f$ be as in Theorem~{\rm \ref{t1}}, $r$ given in {\rm \eqref{e1}} and $k \in \Z_{\geqslant 1}$ fixed.
\begin{enumerate}[$1.$]
   \item Let $\kappa \in \R_{\geqslant 0}$. Then
   $$\sum_{n \leqslant x} \frac{|h_{f,k,r}(n)|}{n^\kappa} \ll \begin{cases} x^{-\kappa + 1/r} (\log x)^r, & \textrm{if\ } 0 \leqslant \kappa < \frac{1}{r}; \\ (\log x)^{r+1}, & \textrm{if\ } \kappa = \frac{1}{r}; \\ 1 , & \textrm{if\ } \kappa > \frac{1}{r}. \end{cases}$$
   \item We also have
   $$\sum_{n > x} \frac{|h_{f,k,r}(n)|}{n} \ll x^{-1+1/r} (\log x)^r.$$
\end{enumerate}
\end{lemma}

\begin{proof}
Follows from Lemma~\ref{le1}, the inequality \eqref{e3} and partial summation.
\end{proof}

\section{$r$-free numbers in short intervals}
\label{s3}

\noindent
The following lemma plays a crucial part in Theorem~\ref{t1}. For a proof, see \cite[Lemma 3.2 and Corollary 5.1]{bor}.

\begin{lemma}
\label{le3}
Let $r \in \Z_{\geqslant 2}$. For any $X \in \R_{\geqslant 1}$ and $0 < Y < X$, define
\begin{equation}
   R_r(X,Y) := X^{\frac{1}{2r+1}} +  Y X^{- \frac{1}{6(4r-1)(2r-1)}} + Y^{1 - \frac{2(r-1)}{r(3r-1)}}. \label{e4}
\end{equation}
\begin{enumerate}[$1.$]
   \item For any $X \in \R_{\geqslant 1}$, $0 < Y < X$ and any $\varepsilon > 0$
   $$\sum_{2Y < n \leqslant 2X} s_r(n) \left( \left \lfloor \frac{X+Y}{n} \right \rfloor - \left \lfloor \frac{X}{n} \right \rfloor \right) \ll_{r, \varepsilon} R_r(X,Y) X^\varepsilon.$$
   \item For any $X \in \R_{\geqslant 1}$, $4^r \leqslant Y < X$ and any $\varepsilon > 0$
   $$\sum_{X < n \leqslant X+Y} \mu_r(n) = \frac{Y}{\zeta(r)} +  O_{r, \varepsilon} \left( R_r(X,Y) X^\varepsilon \right).$$
\end{enumerate}
\end{lemma}

\section{Proof of Theorem~\ref{t1}}
\label{s4}

\noindent
From \eqref{e2}, we get
\begin{eqnarray*}
   \sum_{\substack{x < n \leqslant x+y \\ f(n)=k}} 1 &=& \sum_{d \leqslant x+y} h_{f,k,r}(d) \sum_{\frac{x}{d} < \ell \leqslant \frac{x+y}{d}} \mu_r (\ell) \\
   &=& \left( \sum_{d \leqslant y(y/x)^{1/(2r)}} + \sum_{y(y/x)^{1/(2r)} < d \leqslant 2 y} + \sum_{2y < d \leqslant x+y} \right) h_{f,k,r}(d) \sum_{\frac{x}{d} < \ell \leqslant \frac{x+y}{d}} \mu_r (\ell) \\
   &:=& S_1 + S_2  + S_3.
\end{eqnarray*}
For $S_1$, which will provide the main term, we use the second estimate of Lemma~\ref{le3} giving
\begin{eqnarray*}
   S_1 &=& \sum_{d \leqslant y(y/x)^{1/(2r)}} h_{f,k,r}(d)\left\lbrace \frac{y}{d \zeta(r)} + O \left( R_r \left (\frac{x}{d},\frac{y}{d} \right ) x^\varepsilon \right) \right\rbrace \\
   &=& \frac{y}{\zeta(r)} \sum_{d =1}^\infty \frac{h_{f,k,r}(d)}{d} + O \left( y \sum_{d > y(y/x)^{1/(2r)}} \frac{|h_{f,k,r}(d)|}{d} \right) \\
   & & {} + O \left( x^\varepsilon \sum_{d \leqslant y(y/x)^{1/(2r)}} |h_{f,k,r}(d)| R_r \left (\frac{x}{d},\frac{y}{d} \right )  \right) \\
   &=& \frac{y}{\zeta(r)} H_{f,k,r} (1) + O \left( \left( x^{r-1} y^{r+1} \right)^{\frac{1}{2r^2}} (\log x)^r \right) \\
   & & {} + O \left( x^\varepsilon \sum_{d \leqslant y(y/x)^{1/(2r)}} |h_{f,k,r}(d)| R_r \left (\frac{x}{d},\frac{y}{d} \right )  \right)
\end{eqnarray*}
where we used Lemma~\ref{le2} and where the error term $R_r$ is defined in \eqref{e4}. Using Lemma~\ref{le2} again
\begin{eqnarray*}
   \sum_{d \leqslant y(y/x)^{1/(2r)}} |h_{f,k,r}(d)| R_r \left (\frac{x}{d},\frac{y}{d} \right ) & \ll & x^{\frac{1}{2r+1}} \sum_{d \leqslant y(y/x)^{1/(2r)}} \frac{|h_{f,k,r}(d)|}{d^{\frac{1}{2r+1}}} \\
   & & {} + y x^{- \frac{1}{6(4r-1)(2r-1)}} \sum_{d \leqslant y(y/x)^{1/(2r)}} \frac{|h_{f,k,r}(d)|}{d^{\frac{48r^2-36r+5}{6(2r-1)(4r-1)}}} \\
   & & {} + y^{1 - \frac{2(r-1)}{r(3r-1)}} \sum_{d \leqslant y(y/x)^{1/(2r)}} \frac{|h_{f,k,r}(d)|}{d^{1 - \frac{2(r-1)}{r(3r-1)}}} \\
   & \ll & \left( x^{r-1} y^{r+1} \right)^{\frac{1}{2r^2}} (\log x)^r+ y x^{- \frac{1}{6(4r-1)(2r-1)}} + y^{1 - \frac{2(r-1)}{r(3r-1)}}.
\end{eqnarray*}
Hence
\begin{equation}
   S_1 = \frac{y}{\zeta(r)} H_{f,k,r} (1) + O \left\lbrace x^\varepsilon \left( \left( x^{r-1} y^{r+1} \right)^{\frac{1}{2r^2}} + y x^{- \frac{1}{6(4r-1)(2r-1)}} + y^{1 - \frac{2(r-1)}{r(3r-1)}} \right) \right\rbrace. \label{e5}
\end{equation}
For $S_2$, we use the second point of Lemma~\ref{le2}, so that
\begin{equation}
   |S_2| \ll y \sum_{d > y(y/x)^{1/(2r)}} \frac{|h_{f,k,r}(d)|}{d} \ll \left( x^{r-1} y^{r+1} \right)^{\frac{1}{2r^2}} (\log x)^r. \label{e6}
\end{equation}
Now
$$S_3 = \sum_{2y < d \leqslant x+y} h_{f,k,r}(d) \left( \left \lfloor \frac{x+y}{d} \right \rfloor - \left \lfloor \frac{x}{d} \right \rfloor \right)$$
and using \eqref{e3} and the first estimate of Lemma~\ref{le3} we obtain
\begin{eqnarray}
   |S_3| & \leqslant & \sum_{2y < d \leqslant 2x} s_r(d) \tau(d) \left( \left \lfloor \frac{x+y}{d} \right \rfloor - \left \lfloor \frac{x}{d} \right \rfloor \right) \notag \\
   & \ll & x^\varepsilon \sum_{2y < d \leqslant 2x} s_r(d) \left( \left \lfloor \frac{x+y}{d} \right \rfloor - \left \lfloor \frac{x}{d} \right \rfloor \right) \notag \\
   & \ll & x^{2 \varepsilon} \left( x^{\frac{1}{2r+1}} +  y x^{- \frac{1}{6(4r-1)(2r-1)}} + y^{1 - \frac{2(r-1)}{r(3r-1)}} \right). \label{e7}
\end{eqnarray}
Collecting \eqref{e5}, \eqref{e6} and \eqref{e7} and noticing that 
$$\left( x^{r-1} y^{r+1} \right)^{\frac{1}{2r^2}} \geqslant x^{\frac{1}{2r+1}}$$ 
whenever $y \geqslant x^{\frac{1}{2r+1}}$, we get
$$\sum_{\substack{x < n \leqslant x+y \\ f(n)=k}} 1 = \frac{y}{\zeta(r)} H_{f,k,r} (1) + O_{\varepsilon,r} \left\lbrace x^\varepsilon \left( \left( x^{r-1} y^{r+1} \right)^{\frac{1}{2r^2}} + y x^{- \frac{1}{6(4r-1)(2r-1)}} + y^{1 - \frac{2(r-1)}{r(3r-1)}} \right) \right\rbrace$$
if $x^{\frac{1}{2r+1}} \leqslant y \leqslant 4^{-2r^2} x$. In order to prove the existence of the local density, we generalize \cite[Theorem~1]{ivi}. Every positive integer $n$ may be uniquely written as $n=ab$, with $(a,b)=1$, $a$ $r$-free and $b$ $r$-full. Since $f$ is multiplicative, $f(n) = f(a)f(b) = f(b)$ and hence
\begin{eqnarray*}
    \sum_{\substack{n \leqslant x \\ f(n)=k}} 1 &=& \sum_{\substack{b \leqslant x \\ f(b)=k}} s_r(b) \sum_{\substack{a \leqslant x/b \\ (a,b)=1}} \mu_r(a) \\
    &=& \sum_{\substack{b \leqslant x \\ f(b)=k}} s_r(b) \left\lbrace \frac{x}{\zeta(r) \Psi_r(b)} + O \left( \left( \frac{x}{b} \right)^{1/r} 2^{\omega(b)} \right) \right\rbrace 
\end{eqnarray*} 
where
$$\Psi_r(b) := b \prod_{p \mid b} \left( 1 + \frac{1}{p} + \dotsb + \frac{1}{p^{r-1}} \right).$$
Using the bound
$$\sum_{\substack{b \leqslant x \\ f(b)=k}} \frac{s_r(b) 2^{\omega(b)}}{b^{1/r}} \ll x^\varepsilon \sum_{b \leqslant x} \frac{s_r(b)}{b^{1/r}} \ll x^\varepsilon $$
we get
$$\sum_{\substack{n \leqslant x \\ f(n)=k}} 1 = \frac{x}{\zeta(r)} \sum_{\substack{b \leqslant x \\ f(b)=k}} \frac{s_r(b)}{\Psi_r(b)} + O_{r,\varepsilon} \left( x^{1/r+\varepsilon} \right).$$
Notice that
$$\sum_{\substack{b \leqslant x \\ f(b)=k}} \frac{b s_r(b)}{\Psi_r(b)} \leqslant \sum_{b \leqslant x} s_r(b) \ll x^{1/r}$$
so that the Dirichlet series of the multiplicative function $b \longmapsto \frac{b s_r(b) \mathbf{1}_{f,k}(b)}{\Psi_r(b)} $ is absolutely convergent in the half-plane $\sigma > \frac{1}{r}$. Hence the series
$$\sum_{\substack{b \geqslant 1 \\ f(b)=k}} \frac{s_r(b)}{\Psi_r(b)}$$
converges absolutely, which implies that the limit of 
$$\frac{1}{x} \sum_{\substack{n \leqslant x \\ f(n)=k}} 1$$
exists as $x \to \infty$ and is equal to
$$d_{f,k} = \frac{1}{\zeta(r)} \sum_{\substack{b = 1 \\ f(b)=k}}^\infty \frac{s_r(b)}{\Psi_r(b)} = \underset{s=1}{\res} \left( L \left( s,\mathbf{1}_{f,k} \right)\right)  = \frac{H_{f,k,r} (1)}{\zeta(r)}$$
achieving the proof of Theorem~\ref{t1}.
\qed

\section{Applications}
\label{s5}

\subsection{Abelian groups}

\noindent
As stated in Section~\ref{s1}, the most famous example of prime-independent multiplicative function $f$ satisfying $f(p)=1$ is the arithmetic function $a$ counting the number of finite, non-isomorphic abelian groups of a given order. We have $a \left( p^\alpha \right) = P(\alpha)$ where $P$ is the unrestricted partition function and, from the generating function of $P$, we deduce that
$$L(s,a) = \prod_{j=1}^\infty \zeta(js) \quad \left( \sigma > 1 \right).$$
Hence Theorem~\ref{t1} may be applied with $r=2$ giving the following result.

\begin{corollary}
\label{cor1}
Let $k \in \Z_{\geqslant 1}$ and $x^{\frac{1}{5} + \varepsilon} \leqslant y \leqslant 2^{-16} x$ be real numbers. Then
$$\sum_{\substack{x < n \leqslant x+y \\ a(n)=k}} 1 = d_{a,k} y + O_{\varepsilon} \left\lbrace x^{1/8 + \varepsilon} y^{3/8} + y x^{- 1/42 + \varepsilon} + y^{4/5} x^\varepsilon \right\rbrace.$$
\end{corollary}

\subsection{Plane partitions}

\noindent
Let $P_2(n)$ be the number of plane partitions of $n$ (see \cite{sta} for instance) whose generating function is given by
$$\sum_{n=0}^\infty P_2(n) x^n = \prod_{j=1}^\infty \left( 1-x^j \right)^{-j} \quad \left( |x| < 1 \right).$$
Let $f$ be the multiplicative function such that $f(1)=1$ and $f \left( p^\alpha \right) = P_2(\alpha)$. We deduce from the generating function above that
$$\left( f \left( p^\alpha \right) \right)_{\alpha \in \Z_{\geqslant 0}} = \left( 1,1,3,6,13,24,48,86,160,282,500,859,1479,\dotsc \right)$$
and also
$$L(s,f) = \prod_{j=1}^\infty \zeta(js)^j \quad \left( \sigma > 1 \right).$$
Theorem~\ref{t1} may be applied with $r=2$ again.

\begin{corollary}
\label{cor2}
Let $k \in \Z_{\geqslant 1}$, the function $f$ defined as above and $x^{\frac{1}{5} + \varepsilon} \leqslant y \leqslant 2^{-16} x$ be real numbers. Then
$$\sum_{\substack{x < n \leqslant x+y \\ f(n)=k}} 1 = d_{f,k} y + O_{\varepsilon} \left\lbrace x^{1/8 + \varepsilon} y^{3/8} + y x^{- 1/42 + \varepsilon} + y^{4/5} x^\varepsilon \right\rbrace.$$
\end{corollary}

\subsection{Semisimple rings}

\noindent
Another example, closely related to the function $a$, is the multiplicative function $S$ counting the number of finite, non-isomorphic semisimple rings with a given number of elements. For any prime-powers $p^\alpha$, $S \left( p^\alpha \right) = P^\star (\alpha)$ where $P^\star$ is the number of partitions of $\alpha$ into parts which are square. Since the generating function of $P^\star$ is
$$\sum_{n=0}^\infty P^\star (n) x^n = \prod_{q=1}^\infty \prod_{m=1}^\infty \left( 1-x^{qm^2} \right)^{-1} \quad \left( |x| < 1 \right)$$
we infer that
$$\left( S \left( p^\alpha \right) \right)_{\alpha \in \Z_{\geqslant 0}} = \left( 1,1,2,3,6,8,13,18,29,40,58,79,115,154,213,\dotsc \right)$$
and
$$L(s,S) = \prod_{q=1}^\infty \prod_{m=1}^\infty \zeta (qm^2 s) \quad \left( \sigma > 1 \right).$$

\begin{corollary}
\label{cor3}
Let $k \in \Z_{\geqslant 1}$ and $x^{\frac{1}{5} + \varepsilon} \leqslant y \leqslant 2^{-16} x$ be real numbers. Then
$$\sum_{\substack{x < n \leqslant x+y \\ S(n)=k}} 1 = d_{S,k} y + O_{\varepsilon} \left\lbrace x^{1/8 + \varepsilon} y^{3/8} + y x^{- 1/42 + \varepsilon} + y^{4/5} x^\varepsilon \right\rbrace.$$
\end{corollary}

\subsection{Exponential divisors}

\noindent
A positive integer $d = p_1^{a_1} \dotsb p_s^{a_s}$ is said to be an \textit{exponential divisor} of a positive integer $n= p_1^{\alpha_1} \dotsb p_s^{\alpha_s}$ if and only if, for all $i \in \{1,\dotsb,s \}$, $a_i \mid \alpha_i$. It is customary to denote by $\tau^{(e)} (n)$ the number of exponential divisors of $n$. The function $\tau^{(e)}$ is multiplicative and satisfies $\tau^{(e)} \left( p^\alpha \right) = \tau(\alpha)$. The same is true for the \textit{unitary} exponential divisor function $\tau^{(e) \, \star}$ for which $\tau^{(e) \, \star} \left( p^\alpha \right) = 2^{\omega(\alpha)}$.

\begin{corollary}
\label{cor4}
Let $k \in \Z_{\geqslant 1}$ and $x^{\frac{1}{5} + \varepsilon} \leqslant y \leqslant 2^{-16} x$ be real numbers. If $f = \tau^{(e)}$ or $f=\tau^{(e) \, \star}$
$$\sum_{\substack{x < n \leqslant x+y \\ f(n)=k}} 1 = d_{f,k} y + O_{\varepsilon} \left\lbrace x^{1/8 + \varepsilon} y^{3/8} + y x^{- 1/42 + \varepsilon} + y^{4/5} x^\varepsilon \right\rbrace.$$
\end{corollary}

\subsection{The $r$-th power divisor function}

\noindent
Let $r \in \Z_{\geqslant 2}$ fixed and define the divisor function $\tau^{(r)}$ by $\tau^{(r)}(1)=1$ and, for any $n \in \Z_{\geqslant 2}$
$$\tau^{(r)} (n) = \sum_{d^r \mid n} 1.$$
Then $\tau^{(r)}$ is multiplicative and
$$\tau^{(r)} \left( p^\alpha \right) = 1 + \left \lfloor \frac{\alpha}{r} \right \rfloor \quad \textrm{and} \quad L \left( s,\tau^{(r)} \right) = \zeta(s) \zeta(rs).$$
\begin{corollary}
\label{cor}
Let $k \in \Z_{\geqslant 1}$ and $r \in \Z_{\geqslant 2}$ fixed, and let $x^{\frac{1}{2r+1} + \varepsilon} \leqslant y \leqslant 4^{-2r^2} x$ be real numbers. Then
$$\sum_{\substack{x < n \leqslant x+y \\ \tau^{(r)}(n)=k}} 1 = d_{\tau^{(r)},k} y + O_{r,\varepsilon} \left\lbrace \left( x^{r-1} y^{r+1} \right)^{\frac{1}{2r^2}} x^\varepsilon + y x^{- \frac{1}{6(4r-1)(2r-1)} + \varepsilon} + y^{1 - \frac{2(r-1)}{r(3r-1)}} x^\varepsilon \right\rbrace.$$
\end{corollary}

\end{document}